\documentclass{amsart}
\usepackage {amssymb,latexsym}
\usepackage [all]{xy}
\newtheorem{theorem}{Theorem}[section]
\newtheorem{lemma}[theorem]{Lemma}
\newtheorem {proposition}[theorem]{Proposition}

\theoremstyle{definition}
\newtheorem{definition}[theorem]{Definition}
\newtheorem{example}[theorem]{Example}

\theoremstyle{remark}

\newtheorem{remarks}[theorem]{Remarks}
\numberwithin {equation}{section}

\def\ra{{\rightarrow}}
\def\lra{{\longrightarrow}}

\def\inv{{^{-1}}}
\def\pr{{^{\prime}}}

\def\c{{\mathbb{C}}}
\def\z{{\mathbb{Z}}}
\def\r{{\mathbb{R}}}
\def\q{{\mathbb{Q}}}
\def\qq{{$\q$}}
\def\nn{{\mathbb{N}}}

\def\af{{\mathbb{A}_f}}
\def\zp{{\mathbb{Z}_p}}
\def\qp{{\mathbb{Q}_p}}

\def\cdc{{,\cdots,}}

\def\tqr{{\textquoteright}}

\def\p{{\parallel}}
\def\s{{^{\ast}}}
\def\crs{{^{\times}}}

\def\h{{\mathcal{H}}}

\def\lo{{\mathcal{L}_1}}

\def\ln{{\mathcal{L}_n}}
\def\rcal{{\mathcal{R}}}
\def\gl{{\text{GL}_2^+}}
\def\ff{{\varphi}}
\def\si{{\sigma}}
\def\la{{\lambda}}
\def\lam{{\Lambda}}

\def\al{{\alpha}}

\def\eps{{\mathcal{E}_{\beta}}}
\def\epsin{{\mathcal{E}_{\infty}}}
\def\bt{{\beta}}

\def\kms{{$\text{KMS}_{\beta}$\;}}
\def\cs{{$C^{\ast}$}}
\def\gll{{$\text{GL}_2$}}

\begin{document}
\title{Quantum Statistical Mechanics of $\mathbb{Q}$-lattices and noncommutative geometry}

\author{Vahid Shirbisheh}
\address{Department of Mathematics, Tarbiat Modares University, Tehran, Iran}
\email{shirbisheh@modares.ac.ir}


\begin{abstract}
After recalling some basic notions of quantum statistical mechanics,
we explain the Bost-Connes system that relates the structure of the
maximal abelian extension of \qq \;to the space of \kms states of a
\cs-dynamical system. Afterwards, we study briefly the
Connes-Marcolli $\text{GL}_2$-system as a generalization of the
former system.\\
\end{abstract}
\maketitle 

\section*{introduction}
In [BC], Bost and Connes established a new bridge between number
theory and physics via operator algebras by introducing a quantum
statistical model describing the Galois theory of maximal abelian
extension of $\q$. Their work has motivated further developments in
using noncommutative geometry to tackle the problem of the explicit
class field theory of number fields, to name a few [Co, CM1, CMR1,
HP, J, LLaN], as well as several new directions in the field of
operator algebras, see for example [CuLi, KLnQ, L1,L2, LR1, LR2, T].
Here, we only study some of the developments in the application of
noncommutative geometry to the explicit class field theory problem.
We shall see that the notion of \qq-lattices plays an important role
in new advances, so we focus on this notion. Our study here is far
from being complete or detailed. Therefore, we refer the interested
reader to [CM1, CM2] for more details and more complete lists of
references.

The first draft of this work was prepared when I was a PhD student
at the University of Western Ontario. It was presented as the
required talk for the second part of my PhD comprehensive
examination in May 2005. I would like to thank Matilde Marcolli
whose comments on the first version of this preprint have helped me
to discuss some recent progresses and add some more references in
the present version.
\section{Basics of quantum statistical mechanics}
\begin{definition}
\begin{itemize}
  \item [(a)]
A \textbf{$C\s$-dynamical system} is a pair $(A,\si)$ such that $A$
is a $C\s$-algebra, called \textbf{the algebra of observable}, and
$\si$ is a one parameter group of automorphisms of $A$, called the
\textbf{time evolution of the system}, such that for $s, t \in
\mathbb{R}$, we have $\si_{s+t}=\si_s\si_t$, and $\si$ is strongly
continuous, that is  for every $a \in A$ the map $t\mapsto\si_t(a)$
from $\r$ into $A$ is continuous in norm topology.
  \item [(b)]
A bounded linear functional $\ff:A\ra\c$ is called a \textbf{state
on $A$ }, if
\[
\p \ff \p = 1,
\]
and
\[
\ff(aa\s)\geq0, \qquad \forall a \in A.
\]
\end{itemize}
\end{definition}

\begin{definition}
Let $(A,\si)$ be a $C\s$-dynamical system and $\ff$ be a state on
$A$. For $0<\beta<\infty$, we say $\ff$ is a \textbf{KMS state at
inverse temperature $\beta$ on $(A,\si)$}, or simply a
\textbf{$\text{KMS}_{\beta}$ state on $A$}, if it satisfies
{\bf  $\text{KMS}_{\beta}$ condition}:\\
For every $x, y \in A$, there is a bounded holomorphic function
$F_{x, y}(z)$ on the open strip $0< Im z<\beta$, continuous on the
closure of the strip, such that for all $t \in \mathbb{R}$ we have
\[
F_{x, y}(t)=\ff(x\si_t(y)),
\]
\[
F_{x, y}(t+i\beta)=\ff(\si_t(y)x).
\]
A \textbf{KMS state at $\infty$ on $(A,\si)$} or simply a
\textbf{ground state on $A$} is a weak limit of \kms states $\ff_{\beta}$\tqr s
as $\beta\ra\infty$,
\[
\ff_{\infty}:=lim_{\beta\ra\infty} \ff_{\beta} (a), \qquad
\forall\; a\in A.
\]
\end{definition}
In thermodynamics $\beta=1/kT$, where $k$ is the Boltzmann constant,
and $T$ is the temperature of the system. For simplicity, we assume
$k=1$.

\begin{example}
Let $A=M_n(\c)$ be the algebra of $n\times n$ matrices with
complex entries. Any one parameter group of automorphisms
$(\si_t)_{t \in \mathbb{R} }$ of $A$ has the form
\[
\si_t(x)=e^{itH}x e^{-itH},\qquad \forall x \in, t \in \r,
\]
for some self-adjoint operator $H \in A$. Then for $\beta>0$ one
has a unique \kms state, called the \textbf{Gibbs equilibrium
state}, given by
\begin{eqnarray}
\ff_{\beta}(x)=\frac{\text{Tr}(e^{-\beta H}x)}{\text{Tr}(e^{-\beta
H})},\qquad \forall x\in A.
\end{eqnarray}
The normalizer of above state, i.e.
\[
\text{Tr}(e^{-\beta H})
\]
is called the \textbf{partition function} of the state. Due to the
fact that
\[
\ff_\beta(\sigma_t(x))=\ff_\beta(x), \qquad \forall x\in A, t \in
\r,
\]
the Gibbs state is an equilibrium state with respect to every time
evolution of the system, hence the name.
\end{example}

If one accepts the Gibbs formalism for thermodynamics, namely one
describes a thermodynamics by a \cs-dynamical system $(A,\si)$, then
\kms states on $(A,\si)$ play the role of equilibrium states, that
is for every \kms state $\ff$ we have
\[
\ff(\si_t(a)) = \ff(a),
\qquad \forall \; a \in A,\; t \in \r.
\]
In contrary to the above example, \kms states are not necessarily
unique. To describe the space of \kms states on a \cs-dynamical
system $(A,\si)$, we need the following definition:
\begin{definition}
\begin{itemize}
  \item [(a)] Any compact simplex in a locally convex topological vector space $E$ is
  called a \textbf{Choquet simplex}.
  \item [(b)] Let $M$ be a von Newmann algebra. The center of M is defined as
  $Z(M):=M\cap M\pr$ and M is called a \textbf{factor}, if $Z(M)=1\c$.
  \item [(c)] Let $\ff$ be a state on a \cs-algebra $A$. We say $\ff$ is a
  \textbf{factor state} if $\pi_{\ff}(A)^{\pr\pr}$, the enveloping von Newmann
  algebra of the cyclic representation associated to $\ff$, is a factor.
\end{itemize}
\end{definition}

\begin{proposition} ([BtR]) Let $(A,\si)$ be a \cs-dynamical system
and $\beta>0$. Then the space of \kms states on $(A,\si)$ is a
compact Choquet simplex, and the extreme points are factor states.
\end{proposition}

\noindent\textbf{Notation.} For $0<\beta\leq\infty$, the set of
extreme points of \kms states, which is also called the {\bf set of
extremal \kms states}, is denoted by $\eps$.

\section{The Bost-Connes system, 1-dimensional case}
In [BC], Bost and Connes constructed a \cs-dynamical system
$(A,\si)$, with an action of the id\`{e}les class group of
$\mathbb{Q}$ modulo the connected component of the identity as
symmetry. The Riemann zeta function appears as the partition
function of the system. In this section, we describe the Bost-Connes
(or briefly BC) system.

\begin{definition}
The underlying \cs-algebra $A$ of BC system is generated by two
types of operators $\{e(r); r \in \q/\z\}$, and $\{\mu_n; n \in
\nn\crs\}$ subject to following conditions:
\begin{itemize}
\item[(a)] $\mu_n\mu_n\s=1, \qquad \forall\; n,$
\item[(b)] $\mu_m\mu_n=\mu_n\mu_m, \qquad \forall\; m, n,$
\item[(c)] $e(0)=1, e(r+s)=e(r)e(s),\; \text{and}\; e(r)\s=e(-r), \quad \forall\; r, s,$
\item[(d)] $\mu_ne(r)\mu_n\s=\frac{1}{n}\sum_{ns=r}e(s), \qquad \forall\; n, r.$
\end{itemize}
The time evolution of the system is given by
\begin{eqnarray}
\sigma_t(\mu_n):=n^{it}\mu_n, \quad \sigma_t(e(r)):=e(r).
\end{eqnarray}
\end{definition}

\begin{remarks}
\begin{itemize}
\item[(a)] Bost and Connes defined this \cs-algebra also as the reduced Hecke \cs-algebra of
the Hecke pair $(P^+_{\z}, P^+_{\q})$, where
\[
P^+_{\z}:=\left\{ \begin{array}{c}
             \left[ \begin{array}{cc}
                     1 & n \\
                     0 & 1
                    \end{array}
             \right]
 ; n \in \z
          \end{array}
           \right\},
\]
and
\[
P^+_{\q}:=\left\{ \begin{array}{c}
             \left[ \begin{array}{cc}
                     1 & b \\
                     0 & a
                    \end{array}
             \right]
 ; a, b \in \q, \; \text{and}\; a>0
          \end{array}
           \right\}.
\]

\item[(b)] In [LR], Laca and Raeburn described the above \cs-algebra $A$ as a semigroup \cs-crossed product.
In terms of the above definition, they used Relation (c) to
construct a representation of $C\s(\q/\z)$ in $A$. Then, Relations
(a), (b) define an action of $\nn\crs$ on $A$ by isometries.
Afterwards, Relation (d) shows that $(e,\mu)$ is a covariant pair
for the dynamical system $(C\s(\q/\z), \nn\crs, \beta)$, where
$\beta$ is the action of $\nn\crs$ on $C\s(\q/\z)$ by endomorphisms
which is defined as
\[
\qquad \qquad \beta_n(i(r)):=\frac{1}{n}\sum_{j=1}^n
i(\frac{r}{n}+\frac{j}{n}), \qquad\quad \forall \; n\in \nn\crs,
\; \text{and} \; r\in\q/\z,
\]
where $i:\q/\z\ra C\s(\q/\z)$ is the natural embedding of discrete
group $\q/\z$ in its group \cs-algebra. The existence of a covariant
pair implies the existence of the semigroup \cs-crossed product of
the above semigroup \cs-dynamical system, and
\[
A=C\s(\q/\z)\rtimes_{\beta}\nn\crs.
\]
\end{itemize}
\end{remarks}

\subsection{The Bost-Connes system in terms of 1-dimensional \qq-lattices}\quad
\par
In Remarks 2.2, we saw two different formulations of BC system.
There are yet two more formulations for this system. Its underlying
\cs-algebra can be obtained as a \cs-algebra associated to
1-dimensional \qq-lattices as well as a groupoid \cs-algebra. The
notion of \qq-lattices was first initiated by Connes and Marcolli,
[CM1], in order to generalize BC system to higher dimensions and
find a way to tackle the problem of explicit class field theory of
real quadratic fields. Thereby, Connes-Marcolli (or simply CM)
\gll-system was invented. Afterwards, Connes, Marcolli, and
Ramachandran introduced the notion of a $\mathbb{K}$-lattice for
$\mathbb{K}$ being an imaginary quadratic field in [CMR1], see also
[CM2, CMR2]. In [CMR1], they constructed a quantum statistical
model, similar to BC system, for explicit class field theory of
imaginary quadratic fields. Finally, Ha and Paungam generalized
Bost-Connes-Marcolli system for an arbitrary Shimura datum in [HP].
As one may notice, the key concept in recent developments of the
subject is the notion of \qq-lattices. Therefore, we study it here.
\begin{definition}
\begin{itemize}
\item[(a)] An (n-dimensional) \textbf{$\q$-lattice in $\r^n$} is a pair $(\Lambda,\ff)$,
where $\Lambda$ is an $n$-dimensional lattice, that is a discrete
subgroup of $\r^n$ of rank $n$, or equivalently, a free subgroup
generated by $n$ linearly independent vectors, and
\[ \ff:\q^n/\z^n\longrightarrow\q\Lambda/\Lambda
\]
is a homomorphism of abelian groups. We denote the space of
$n$-dimensional \qq-lattices by $\ln$.
\item[(b)] Two $\q$-lattices $(\Lambda_1,\ff_1)$, and
$(\Lambda_2,\ff_2)$ are called \textbf{commensurable}, and denote
by $(\Lambda_1,\ff_1)\sim(\Lambda_2,\ff_2)$, if
\[
\q\Lambda_1=\q\Lambda_2,
\]
and
\[
(\ff_1-\ff_2)(x)\in\Lambda_1+\Lambda_2, \qquad \forall\; x\in\q^n/\z^n.
\]
\end{itemize}
\end{definition}
The latter condition can also be read as $\ff_1$ and $\ff_2$ are
equal modulo $\Lambda_1+\Lambda_2$.
\begin{lemma} The commensurability of \qq-lattices is an
equivalence relation.
\end{lemma}
\begin{proof}
We only show that commensurability is transitive. Let
$(\Lambda_1,\ff_1)\sim(\Lambda_2,\ff_2)$, and
$(\Lambda_2,\ff_2)\sim(\Lambda_3,\ff_3)$. Obviously
$\q\lam_1=\q\lam_3$.\\
Let $\lam=\lam_1+\lam_2+\lam_3$, then $(\ff_1-\ff_3)(x)\in\lam$ for all $x\in\q/\z$. To see that it actually belongs to $\lam_1+\lam_3$ we need to show that $\lam_1+\lam_3$ is of finite index in $\lam$.\\
$\q\lam_1=\q\lam_2$ implies that $m\lam_2\subset\lam_1$ for some $m\in\nn$. Thus
$m\lam_2+\lam_1+\lam_3\subset \lam_1+\lam_3$. Let $e_1\cdc e_n$
generate $\lam_2$, then $\{\sum_{i=1}^n m_ie_i+\lam_1+\lam_3;
1\leq m_i\leq m\}$ is a complete set of cosets of $\lam_1+\lam_3$
in $\lam$, so $\lam_1+\lam_3$ is of finite index in $\lam$.\\
Thus, there exists $M\in\nn$ such that for all $x\in
\q^n/\z^n$
\[
M(\ff_1-\ff_3)(x)\in\lam_1+\lam_3.
\]
This implies that
\[
(\ff_1-\ff_3)(x)=M(\ff_1-\ff_3)(\frac{x}{M})\in\lam_1+\lam_3,
\qquad\forall\;x\in\q^n/\z^n.
\]
\end{proof}
\noindent\textbf{1-dimensional \qq-lattices.} One easily checks
that every \qq-lattice $(\lam,\ff)$ in $\r$ can be written as
$(\la\z,\la\rho)$ for some $\la>0$ and some
\begin{eqnarray}
\rho\in R:=\text{Hom}_{\z}(\q/\z,\q/\z).
\end{eqnarray}
Thus, the space $\lo/\r^+$ of 1-dimensional \qq-lattices up to
scaling can be identified by $R$. On the other hand, one can
formulate the equivalence classes of 1-dimensional \qq-lattices
under the equivalence relation of commensurability as the orbits of
an action of $\nn^{\times}$ on the space of \qq-lattices.

\begin{lemma}
Two \qq-lattices $(\la_i\z,\la_i\rho_i),\; i=1,2$ are
commensurable if and only if there exist $m, n \in \nn$ such that
\[
m\la_1=n\la_2,
\]
and
\[
(n\rho_1-m\rho_2)(x)\in \z, \qquad \forall\; x\in \q/\z.
\]
\end{lemma}

Therefore, the orbit space of the action  $n(\rho):=n\rho$ on
$\lo/\r^+$ is \textbf{the space of commensurability classes of
1-dimensional \qq-lattices up to scaling}. To study this space as a
noncommutative quotient space, we consider the semigroup \cs-crossed
product of the induced action on the algebra of continuous complex
functions on $R$.
\begin{definition}
Define $\al:\nn\crs \lra \text{Aut}(C(R))$ by
\begin{eqnarray}
\al_nf(\rho):=\left\{ \begin{array} {r@{\qquad \text{if}\quad}l}
                         f(n^{-1}\rho) & \rho\in nR\\
                         0 & \text{otherwise.}
                      \end{array} \right.
\end{eqnarray}
The semigroup \cs-crossed product $C(R)\rtimes_{\al}\nn\crs$ of the
above action is \textbf{the noncommutative quotient space of
1-dimensional \qq-lattices up to scaling by the equivalence relation
of commensurability}.
\end{definition}

\begin{remarks}
\begin{itemize}
\item[(a)] Let $\af$ denote the ring of finite ad\`{e}les on $\q$,
namely
\[
\af=\prod_{\text{res}}\qp:=\bigcup_F \; \prod_{p\in F}\qp \times
\prod_{p\notin F} \zp,
\]
where the union is taken over all finite families of rational prime
numbers and $\qp$ (resp. $\zp$) is the field of $p$-adic numbers
(resp. the ring of $p$-adic integers ). The topology of $\af$ is
defined such that each set in the above union is an open set. Then
\[
\widehat{\z}:=\prod_{p\;\text{prime}}\zp
\]
is the maximal compact subring of $\af$, and it is shown in [W] that
$\af/\widehat{\z}$ and $\q/\z$ are isomorphic as abelian groups.
This isomorphism gives rise to the following isomorphism
\[
j: \widehat{\z} \ra R, \; j(a)(x):= ax,\quad\forall
x\in\af/\widehat{\z},\; \forall a\in \widehat{\z}
\]
Therefore, we can consider R as a compact abelian group.
\item[(b)]The following map shows the Pontrjagin duality between $\q/\z$ and $R$.
\[
e:\q/\z\ra R, \quad e(r)(\rho):= e^{2\pi i \rho(r)}, \; \forall\; r \in \q, \;\forall\; \rho \in R.
\]
\item[(c)] Under the above duality $\al$, the action of $\nn\crs$ on $C(R)$, corresponds to $\beta$, the action of $\nn\crs$ on $C\s(\q/\z)$, that is the following diagram  is commutative.
\[
\xymatrix{
C\s(\q/\z) \ar[d]^{\bt}\ar[r]^{\Gamma}&C(R)\ar[d]^{\al} \\
C\s(\q/\z)\ar[r]^{\Gamma}&C(R)
}
\]
where $\Gamma$ is the Gelfand transform, i.e.
$\Gamma(i(r))(\rho)=\rho(r)$, for $r\in\q/\z$. Thus, we have the
following isomorphism:
\[
C(R)\rtimes_{\al}\nn\crs\simeq C\s(\q/\z)\rtimes_{\beta}\nn\crs.
\]
\end{itemize}
\end{remarks}
We summarize the above remarks as the following theorem:

\begin{theorem}
The \cs-algebra of the BC system can be realized as the
noncommutative quotient space of \qq-lattices up to scaling under
the equivalence relation of commensurability.
\end{theorem}

\subsection{Groupoid approach to the Bost-Connes system}\quad
\par
Now we construct the groupoid $G$ of the equivalence relation of
commensurability on 1 dimensional \qq-lattices up to scaling. The
groupoid \cs-algebra of this groupoid is another description of the
Bost-Connes \cs-algebra.
\par
The groupoid $G$ is defined by
\[
G:=\{(r,\rho)\in \q^+ \times R ; r \rho \in R \},
\]
with the composition defined for two elements of $G$ by
\[
(r_1,\rho_1)\circ(r_2,\rho_2):=(r_1r_2,\rho_2), \quad \text{if} \;
r_2\rho_2=\rho_1,
\]
and the source and the target are given by
\begin{eqnarray*}
s:G \ra R, \quad (r,\rho)\mapsto \rho,\\ t:G \ra R,
\quad(r,\rho)\mapsto r\rho.
\end{eqnarray*}
The convolution product on $C(G)$, the algebra of complex
continuous functions on $G$, is defined by
\[
f_1*f_2(r,\rho):=\sum_{s\rho\in R} f_1(rs^{-1}, s\rho)f_2(s,\rho).
\]
The above sum is finite, because if $\frac{a}{b} \in \q/\z$, then
$\rho(\frac{a}{b})\neq 0$ implies $\frac{1}{b}\rho\notin R$, and $Z=
\{\frac{a}{b}\in\q/\z; \rho(\frac{a}{b})=0\}$ is finite for every
$\rho\neq 0$. The involution on $C(G)$ is defined by
\[
f\s(r,\rho):=\overline{f(r^{-1},r\rho)}.
\]
\begin{proposition} ([CM1]) Let $\mathcal{R}_1$ denote the groupoid of
the equivalence relation of commensurability in 1-dimensional
\qq-lattices. The map
\[
\eta(r,\rho) = ((r^{-1}\z,\rho),(\z,\rho)), \quad \forall\;
(r,\rho)\in G,
\]
defines an isomorphism of locally compact \'etale groupoids
between $G$ and the quotient $\mathcal{R}_1/\r_+^{\ast}$ of the
equivalence relation of commensurability on the space of
1-dimensional \qq-lattices by the natural scaling action of
$\r_+^{\ast}$.
\end{proposition}
The above proposition allows us to consider BC \cs-algebra as
$C^{\ast}(G)$, the groupoid \cs-algebra of $G$.\\

\subsection{The Bost-Connes system and number theory}\quad
\par

The following theorem is the main result of Bost and Connes in [BC], which describes the space of \kms state for all $\beta$\tqr s.
\begin{theorem}
\begin{itemize}
\item[(a)] In the range $0<\beta\leq 1$ there is a unique \kms state. Its restriction to $\q[\q/\z]$, the image of $\q/\z$ in BC \cs-algebra, is of the form
\[
\ff_{\bt}(e(a/b))= b^{-\bt} \prod_{p \; \text{prime},\; p\mid b} \frac{1-p^{\bt-1}}{1-p^{-1}}.
\]
\item[(b)] For $1<\bt\leq\infty$ $\eps$, the extreme \kms states, are parameterized by embeddings $\iota:\q^{ab}\hookrightarrow \c$, and on $\q[\q/\z]$ we have
\[
\ff_{\bt, \iota}(e(a/b)) = Z(\bt)^{-1}\sum_{n=1}^{\infty} n^{-\bt} \iota(\zeta^n_{a/b}),
\]
where the partition function $Z(\bt) = \zeta(\bt)$ is the Riemann zeta function and $\zeta_{a/b}$ is the root of unity associated to $a/b$.
\item[(c)] For $\bt=\infty$, every $\ff\in\epsin$ maps $\q[\q/\z]$ into $\q^{ab}\subset\c$,
and the Galois group $Gal(\q^{ab}/\q)$ acts on the values of states
in $\epsin$ restricted to $\q[\q/\z]$. Then, the class field theory
isomorphism $\theta:\text{Gal}(\q^{ab}/\q)\ra R\s$ intertwines the
actions of the Galois group with the action of $R\s$ by symmetries
of $\q[\q/\z]$, that is
\[
\gamma(\ff(x))=\ff(\theta(\gamma)(x)),
\]
where $\gamma \in \text{Gal}(\q^{ab}/\q),\; \ff\in\epsin,\; x\in
\q[\q/\z]$, or equivalently, the following diagram commutes.
\[
\xymatrix{ \q[\q/\z]
\ar[d]^{\theta(\gamma)}\ar[r]^{\ff}&\q^{ab}\subset\c\ar[d]^{\gamma}
\\ \q[\q/\z]\ar[r]^{\ff}&\q^{ab}\subset\c }
\]
\end{itemize}
\end{theorem}

\noindent\textbf{The explicit class field theory of \qq and the BC
system.}
\par
The main theorem of the class field theory is the following
isomorphism for any number field $K$, a finite extension of \qq,
\[ \theta:Gal(K^{ab}/K)\ra\frac{C_K}{D_K},
\]
where $K^{ab},\; C_K$, and $D_K$ are respectively the maximal
abelian extension of $K$, the group of id\`ele classes of $K$, and
the connected component of the identity in the group of id\`ele
classes.\\ \par
 A theorem of Kronecker and Weber states
$\q^{ab}$ is isomorphic to $\q^{cycl}$ the cyclotomic extension of
\qq, the extension obtained by adding all roots of unity to \qq, and
$C_K/D_K$ is isomorphic to $R\s$, the group of of invertible
elements of $R$. Thus, the elements of $\q/\z$ are the generators of
$\q^{ab}$, and the above theorem illustrates the action of the
Galois group $Gal(\q^{ab}/\q)$ on the generators of $\q^{ab}$
explicitly in terms of the action of $R\s$ on $\q[\q/\z]$ via the BC
\cs-dynamical system and its ground states. Surprisingly, the
explicit description of the generators of $K^{ab}$ and the action of
the Galois group $Gal(K^{ab}/K)$ has been done only for $\q$ and
imaginary quadratic fields $\q(\sqrt{-d})$, while the similar
description for more general number fields, in particular real
quadratic fields, is the subject of Hilbert's 12th problem. As we
mentioned before, the case of imaginary quadratic fields was treated
in [CMR1]. The above theorem brought some hopes to find an answer to
this problem at least for real quadratic fields $\q(\sqrt{d})$ using
noncommutative geometry. Indeed, the Connes-Marcolli \gll-system is
an attempt towards this goal.


\section{The Connes-Marcolli system, the 2-dimensional case}
The observation summarized in Remarks 2.7 led Connes and Marcolli,
[CM1], to consider the noncommutative quotient space of
2-dimensional \qq-lattices up to scaling and equivalence relation of
commensurability as the generalization of BC system. Similar to the
description of BC \cs-algebra as a groupoid \cs-algebra, Connes and
Marcolli described 2-dimensional system as the completion of the
convolution algebra over $\rcal_2/\c\s$, where $\rcal_2$ is the
groupoid of the equivalence relation of commensurability on
2-dimensional \qq-lattices, and $\c^*$ represents the natural
scaling
in $\r^2$ considered by non-zero complex numbers. In this section, we review these results briefly.\\
\par
An arbitrary 2-dimensional \qq-lattice can be written in the form
\[
(\lam, \ff) = (\la(\z+\z\tau),\la\rho),
\]
 for some $\la\in \c\s$, $\tau\in\mathbb{H}$, $\rho\in M_2(R)=\text{Hom}(\q^2/\z^2,\q^2/\z^2)$, where
$\mathbb{H}=\{x+iy\in \c; y>0\}$ is the upper half plane. In order
to define the action of $\gl(\r)$, we choose a basis $\{e_1=1, e_2=i
\}$ of $\c$ as a 2-dimensional real vector space. We set $\lam_0:=\z
e_1+\z e_2$. Then every element $\rho \in M_2(R)$ defines a
homomorphism
\[
\rho:\q^2/\z^2\ra\q\lam_0/\lam, \quad\rho(a) = \rho_1(a)e_1 +
\rho_2(a)e_2.
\]
Let $\Gamma = \text{SL}(2,\z)$. We define the action of
$\Gamma\times\Gamma$ on the space
\[
\tilde{\mathcal{U}}:= \{ (g,\rho,\al)\in \gl(\q)\times
M_2(R)\times\gl(\r)\ ; g\rho\in M_2(R)\}
\]
by
\begin{eqnarray}
(\gamma_1,\gamma_2)(g,\rho,\al)=(\gamma_1 g \gamma_2^{-1},
\gamma_2\rho, \gamma_2\al).
\end{eqnarray}
The following proposition is the analogue of Proposition 2.8 for the
2-dimensional case.

\begin{proposition} ([CM1]) Let $\mathcal{R}_2$ denote the groupoid of the equivalence relation
of commensurability on 2-dimensional \qq-lattices. $\mathcal{R}_2$
can be parameterized by the quotient of $\tilde{\mathcal{U}}$
under the action of $\Gamma\times\Gamma$ via the map
\[
\eta :
\frac{\tilde{\mathcal{U}}}{\Gamma\times\Gamma}\longrightarrow\mathcal{R}_2
\]
\[
[(g,\rho,\al)] \mapsto ((\al\inv g\inv \lam_0,
\al\inv\rho),(\al\inv\lam_0,\al\rho)).
\]
\end{proposition}
Since two \qq-lattices $(\la_k,\ff_k), k=1, 2$ are commensurable
if and only if for any $\la \in \c\s$, $(\la\lam_k,\la\ff_k)$ are
commensurable, we should consider $\mathcal{R}_2/\c\s$, where the
action of $\c\s$ is defined by the embedding $\c\s$ into $\gl(R)$
by
\begin{eqnarray}
a+ib\in\c\s\mapsto \left(\begin{array}{cc} a & b \\ -b & a
\end{array} \right) \in \gl(\r).
\end{eqnarray}
However, this action is not free, so $\mathcal{R}_2/\c\s$ is not a
groupoid anymore. But one still can define a convolution product on
$C_c(\mathcal{R}_2/\c\s)$, the algebra of continuous complex
functions on $\mathcal{R}_2/\c\s$ with compact support. First, we
observe that (3.2) gives rise to the following isomorphism:
\begin{eqnarray}
\frac{\gl(\r)}{\c\s}\longrightarrow\mathbb{H}
\end{eqnarray}
\[
\left(\begin{array}{cc} a & b\\ c & d \end{array} \right) \mapsto
\frac {ai+b}{ci+d}.
\]
Thus, we can identify $\mathcal{R}_2/\c\s$ with the quotient of the
space
\[
\mathcal{U}:=\{  (g,\rho,\al)\in \gl(\q)\times
M_2(R)\times\mathbb{H}\ ; g\rho\in M_2(R)\}
\]
under the action of $\Gamma\times\Gamma$ defined in (3.1). Now, one
may consider elements of $C_c(\mathcal{R}_2/\c\s)$ as those
continuous complex functions on $\mathcal{U}$ that are invariant
under the action $\Gamma\times\Gamma$. Then, the convolution product
is defined by
\[
f_1*f_2(g, \rho, z):=\sum_{s\in \Gamma\setminus \gl(\q);\;
s\rho\in M_2(R) } f_1(gs\inv, s\rho, s(z))f_2(s, \rho, z),
\]
and the involution is defined by
\[
f\s(g, \rho, z):=\overline{f(g\inv, g\rho, g(z))}.
\]
Afterwards, the above algebra is represented as a subalgebra of
operators on a Hilbert space $\h$. The completion of the latter
algebra is the underlying \cs-algebra of the Connes-Marcolli
$\text{GL}_2$-system $A_2$ and the time evolution is given by
\[
\si_t(f)(g,\rho,z):=\text{det}(g)^{it} f(g,\rho,z)
\]
\par

As before, Let $\Gamma=\text{SL}(2,\z)$ and
$\widehat{\z}=\prod_{p\;\text{prime}}\zp$. Let $\Gamma$ act on
$\mathbb{H}$ via linear fractional transformations, that is, for
$g=\left(\begin{array}{ll} a&b\\c&d\end{array}\right)\in \Gamma$ and
$z\in \mathbb{H}$, we have $g(z)=\frac{az+b}{cz+d}$. Moreover, let
$\Gamma$ act on $\widehat{\z}$ componentwise, i.e. for $(m_p)_p\in
\widehat{\z}$ and $g\in \Gamma$, we have $g((m_p)_p)=(gm_p)_p$.
Combining these two actions componentwise, we obtain an action of
$\Gamma$ on $\mathbb{H}\times \text{GL}_2(\widehat{\z})$. Then, the
space of \kms states of the Connes-Marcolli $\text{GL}_2$-system is
determined as follows:
\begin{theorem}
The \kms states of the $\text{GL}_2$-system are characterized as
follows:
\begin{itemize}
\item[(a)] For $\bt < 1$ there are no \kms states.
\item[(b)] For $1<\bt \leq 2$ there is a unique \kms state.
\item[(c)] For $\bt>2$ there is a one-to-one affine correspondence between \kms states
and probability measures on $\Gamma \setminus (\mathbb{H}\times
\text{GL}_2(\widehat{\z}))$. In particular, extremal \kms states are
in bijection with $\Gamma$-orbits in $\mathbb{H}\times
\text{GL}_2(\widehat{\z})$
\end{itemize}
\end{theorem}
Some parts of the above theorem was proved in [CM1]. For the
complete proof and more discussions on the above theorem and the
case $\beta=1$ see Theorems 3.7 and 4.1 as well as Remarks 3.8 and
4.8 of [LLaN].
\bibliographystyle {amsalpha}


\end{document}